\documentclass[a4paper,10pt]{article}
\usepackage[utf8]{inputenc}
\usepackage{amsfonts}
\usepackage{amsmath}
\usepackage{amsthm}
\usepackage{verbatim} 
\usepackage{hyperref} 

\usepackage{color}
\usepackage{tikz}
\newcommand{\red}[1]{\textcolor{red}{\bf #1}}

\newcommand{\CC}{\mathbb{C}}
\newcommand{\PP}{\mathbb{P}}
\newcommand{\DD}{\mathbb{D}}
\newcommand{\calC}{\mathcal{C}}
\newcommand{\calN}{\mathcal{N}}
\newcommand{\calG}{\mathcal{G}}

\newcommand{\calS}{\mathcal{S}}
\newtheorem{theorem}{Theorem}[section]
\newtheorem{lemma}[theorem]{Lemma}
\newtheorem{cor}[theorem]{Corollary}
\newtheorem{prop}[theorem]{Proposition}
\newtheorem{defn}[theorem]{Definition}
\newtheorem{remark}[theorem]{Remark}
\newtheorem{ex}[theorem]{Example}
\newtheorem*{question}{Question}

\title{Irreducibility of analytic arc-sections of hypersurface singularities.}
\author{Miguel Angel Marco-Buzunariz \footnote{Partially supported by MTM2016-76868-C2-2-P and Grupo ``Investigación en Educación Matemática'' of Gobierno de Aragón/Fondo Social Europeo.}, Maria Pe Pereira\footnote{Partially supported by MTM2017-89420-P, MTM2016-76868-C2-1-P and ERC Consolidator Grant NMST.}}

\begin{document}

\maketitle

\begin{abstract}
We explore the existence of irreducible and reducible arc-sections in an irreducible hypersurface singularity germ along finite projections. 

In particular we provide examples of irreducible isolated hypersurface
singularities for which no irreducible arc-sections exist, and show that
reducible ones always exist. Moreover, we give an algorithm to check if a given
projection allows irreducible arc-sections, and find them if they exist.


\end{abstract}
\section{Introduction}\label{sec:intro}
Given a hypersurface germ singularity $(X,\bar{0})\subseteq (\mathbb{C}^{n+1},\bar{0})$ defined by the equation $F(x_1,\ldots, x_{n+1})=0$, and an injective arc
$\gamma:(\mathbb{C},0)\to (\mathbb{C}^n,\bar{0})$, we say that the plane curve
defined by $F(\gamma(t),x_{n+1})=0$ is an \emph{arc-section} of $(X,\bar{0})$.

Consider, for example, a suspension singularity given by an equation of the form $z^m=f(x,y)$ with $f(x,y)$ an irreducible germ of order $d$ with $g.c.d.(m,d)=1$. 
Given any generic line $L$ in the $XY$-plane parametrized as $(at, bt)$, we consider the plane curve germ  given by $z^m=f(at, bt)$. It is easy to see that this plane branch is irreducible since the associated Newton polygon consists on a unique bounded segment with no integer points in its interior. Note that, in the case $ord(f)>m$, we have that the projection direction is not contained in the tangent cone of the surface at the origin viewed as a set in $\CC\PP^2$.


We wonder if this is the general case. That is:

\begin{question} Let $(X,\bar{0})$ be an irreducible hypersurface germ defined in $\CC^{n+1}$ by the equation $F(x_1,...,x_{n+1})=0$. Do there exist a linear projection $\pi:\CC^{n+1}\to \CC^n$, say $(x_1,...,x_{n+1})\mapsto (x_1,..,x_n)$ with $\pi\mid_{(X,\bar{0})}$ finite and an injective arc $\gamma:(\CC,0)\to (\CC^n,\bar{0})$ such that $F(\gamma(t),x_{n+1})=0$ is an irreducible plane curve germ?
\end{question}
This was privately formulated by A. Ploski and E. Garc\'ia Barroso to the second author. We study it here for the first time, developping some tools to study variations of the question in the surface case.

We give a negative answer to the question in Proposition~\ref{cor:ex}, by
showing a family of counterexamples. Furthermore, in
Theorem~\ref{thm:existfullyreducedsection} we show that arcs where the curve is
reducible always exist. Section~\ref{sec:alg} gives a method to determine if
arcs that produce irreducible sections do exist or not (and find them in case
they do).

Note that the existence of irreducible arc-sections in a hypersurface $(X,\bar{0})$ over arcs not tangent to the discriminant of a projection implies that $(X,\bar{0})$ is irreducible.

The article is structured as follows. Sections 2 and 3 fix the notation. In
Section 4 we give necessary conditions on the tangent cone of $(X,\bar{0})$ for
irreducible arc-sections to exist (see Lemma \ref{lem:L}). We also give
examples where they do never exist for any projection (see Example
\ref{ex:ce0}).
These are also examples of irreducible hypersurfaces $(X,\bar{0})$ whose intersection
with any smooth surface is reducible.

In Section \ref{sec:red} we see that reducible arc-sections always exist for every finite projection $\pi|_{X,\bar{0}}$ , and moreover, for every finite projection $\pi|_{X,\bar{0}}$, there are arc-sections with as many components as the degree of $\pi|_{X,\bar{0}}$.

In subsection \ref{sec:gen}, we see that, apart from the cases where the
tangent cone of $(X,\bar{0})$ is a union of lines, the generic arc is never
irreducible. In this case, we see that the arc with irreducible arc-section has
to be tangent to the discriminant. In Section \ref{sec:disc}, we give an
alternative description of the tangent cone of the discriminant in many cases
which can be usefull in order to simplify the search for irreducible arc-sections.

In Section \ref{sec:alg}, we give an algorithm  to determine whether irreducible arc-sections over an arc not contained in the discriminant exists or not; and find them in case they do.

\section{Notation and general setting}\label{sec:not}

Let $(X,\bar{0})$ be a hypersurface germ  defined by $F(x_1,...,x_{n+1})=0$.
Let $m$ be the multiplicity of $(X,\bar{0})$. Let $\pi:\CC^{n+1}\to \CC^n$ be a
linear projection such that $\pi|_{(X,\bar{0})}$ is finite (see Remark
\ref{re:finite}). After a linear change of coordinates we can assume that $\pi$
is given by $$(x_1,...,x_{n+1})\mapsto(x_1,..,x_n).$$
We say the projection is \emph{transverse} (with respect to $(X,\bar{0})$) if
it is in a direction not contained in the tangent cone of $(X,\bar{0})$. It is
\emph{non-transverse} in the other case.

Let $\Delta$ denote the discriminant of $\pi\mid_{X}$. Note that the restriction $\pi|_{X\setminus \pi^{-1}(\Delta)}$ is a regular covering of certain degree $d$. Then, given a loop $\delta$ in $\CC^{n}\setminus\Delta$ we have the monodromy of the covering along $\delta$ which gives an element of $Sym(d)$, the permutation group of $d$ points.

Since $\pi|_{X,\bar{0}}$ is finite, we can assume $F(x_1,...,x_{n+1})$ is a Weierstrass polynomial in $x_{n+1}$ of degree $d$. Moreover, if the projection is transverse then $d=m$. In the other case, we have $d>m$.

We denote by $C_{X,\bar{0}}$ the tangent cone of $(X,\bar{0})$ seen, either as a conic subvariety germ of $(\CC^{n+1},\bar{0})$ or a subvariety of $\CC\PP^n$.

Given $\pi:\CC^{n+1}\to \CC^n$ as above,  we  denote by $p_\pi$  the point in $\CC\PP^n$ that gives the projection direction. We denote by $\overline{\pi}:\CC\PP^n\to \CC\PP^{n-1}$ the projection from $p_\pi$.

Given a subspace $V$ of $\CC\PP^n$, we denote by $<V>$ the linear subspace of $\CC^{n+1}$ that it generates.


\section{Arc-sections.}\label{sec:arc}

In this section we  fix notation and definitions regarding arcs and arc-sections.

\begin{defn} An \emph{arc} in $(\CC^n,\bar{0})$ is a holomorphic mapping germ
$\gamma:(\CC,0)\to (\CC^n,\bar{0})$. A \emph{representative} of an arc $\gamma$ is
the restriction of $\gamma$ to some closed disk $\DD\subset\CC$ centered at
$0\in\CC$ contained in the domain of $\gamma$.  A \emph{smooth arc} is an arc
$\gamma$ such that $\gamma'(0)\neq \bar{0}$. We say that an arc $\gamma(t)$ is
\emph{transversal} to a subvariety $Y$ of $(\CC^n,\bar{0})$ at $t_0\in \CC$ if
$\gamma'(t_0)\neq \bar{0}$ and $\gamma'(t_0)$ is not contained in the tangent
space of $Y$ at $\gamma(t_0)$. We say that $\gamma(t)$ is \emph{transversal} to
$Y$ if it is transversal at any point.
\end{defn}

In general we will only consider injective representatives $\gamma|_\DD$ (which
always exist for generically 1:1 arcs). In particular we will assume that
$\DD$ is small enough so that $(\gamma|_{\DD})^{-1}(\bar{0})=\{0\}$. Then, note
that $\gamma|_{\partial \DD}$ is a simple loop in $\CC^n\setminus\{\bar{0}\}$.



\begin{lemma}\label{lema-equiv-irred-monod} Let $(X,\bar{0})$, $F(x_1,...,x_{n+1})\in \CC\{x_1,...,x_{n+1}\}$ and $\pi$ be as above. Given an arc $\alpha:(\CC,0)\to (\CC^n,\bar{0})$ whose image is not completely contained in the discriminant $\Delta$, the following are equivalent:
\begin{enumerate}
\item $F(\alpha(t),x_{n+1})=0$ defines an irreducible germ  at $(0,0)\in \DD\times \CC$
\item $\pi^{-1}(\alpha(\DD))\cap X$ is an irreducible germ at $\bar{0}\in \CC^{n+1}$ for an injective representative $\alpha:\DD\to \CC^n$ such that $\alpha(\DD)\cap \Delta=\{\bar{0}\}$
\item the monodromy along $\alpha(\partial \DD)$ is transitive (with $\alpha|_{\DD}$ a representative as in (2)), that is, it gives a $d$-cycle of $Sym(d)$.
\end{enumerate}
\end{lemma}
\begin{proof}
To see $(1)\Leftrightarrow(2)$ we observe that the mapping $$\CC^2\to \CC^{n+1}$$
$$(t,x_{n+1})\mapsto (\alpha(t),x_{n+1})$$ is injective restricted to $\DD\times \CC$ since $\alpha|_\DD$ is injective. Then, its restriction to the germ $F(\alpha(t),x_{n+1})=0$ is a homeomorphism onto $\pi^{-1}(\alpha(\DD))\cap X$. Then, one of them is irreducible if and only if so is the other.

To see $(2)\Leftrightarrow(3)$ we note that the singular locus of $\pi^{-1}(\alpha(\DD))\cap X$ is just the origin. Then, $\pi^{-1}(\alpha(\DD))\cap X$  is irreducible if and only if it remains connected after removing the origin. The map $\pi\mid_{\pi^{-1}(\alpha(\DD\setminus \{0\}))\cap X} $ onto $\alpha(\DD\setminus \{0\})$ is a cover over a space that can be retracted to the simple closed loop $\alpha(\partial\DD)$. So $\pi^{-1}(\alpha(\DD\setminus \{0\}))\cap X$ is connected if and only if the monodromy along $\alpha(\partial\DD)$ is transitive.
\end{proof}
Note that equivalence of (1) and (2) is also true if the arc is contained in the discriminant.

Last lemma justifies the following definition and allows us to think  geometrically about the problem.

\begin{defn} Let $\pi:\CC^{n+1}\to \CC^n$ be a linear projection such that $\pi|_{(X,\bar{0})}$ is finite. An \emph{arc-section} of $(X,\bar{0})$ is 
a section of $(X,\bar{0})$ of the form $\pi^{-1}(\alpha(\DD))\cap X$ for certain injective arc representative $\alpha:\DD\to \CC^n$ that is either contained in the discriminant $\Delta$ of $\pi|_{X}$ or $\alpha(\DD)\cap \Delta=\{\bar{0}\}$.

\end{defn}


So, our main question is about the existence or non-existence of irreducible and reducible arc-sections for irreducible hypersurface germs. 

Note that if $\alpha(\DD)$ is contained in $\Delta$, then there are embedded components in the arc-section. So, for arcs in $\CC^n$ whose image is contained in the discriminant, we ask whether they give irreducible arc-sections considered with its reduced structure.

\begin{remark}\label{re:finite} As we said, we always assume $\pi|_{X,\bar{0}}$
is finite/proper, which is equivalent to $\pi^{-1}(\bar{0})=\{\bar{0}\}$. If
the projection is not proper, then  an arc-section either has
$\pi^{-1}(\bar{0})$ as irreducible component or is a (cylindrical) surface.
Note that a non-trasverse linear
projection, that is a projection in a direction contained in the tangent cone
of $(X,\bar{0})$, is always proper in case this line of the tangent cone is not
contained in $X$.
\end{remark}

\section{Looking at the tangent cone of an arc-section.}\label{sec:tan}

Here we will see some necessary conditions on the tangent cone of $(X,\bar{0})$ for the existence of irreducible arc-sections.


\begin{lemma}\label{lem:L}
Let $\pi:\CC^{n+1}\to \CC^n$ be a linear projection such that $\pi|_{(X,\bar{0})}$ is finite.
Let $L$ be a line in $\CC\PP^n$ containing the projection direction. Assume the intersection (with reduced structure) of $L$ and $C_{X,\bar{0}}$ in $\CC\PP^n$ is a union of  isolated points. Consider any arc in $(\CC^n,\bar{0})$ with tangent at the origin in the direction given by a point of $L$. Then the reduced tangent cone of the corresponding arc-section  is given by the intersection $C_{X,\bar{0}}\cap L$. In particular, if the intersection is 2 or more different points then the arc section is reducible.
\end{lemma}

\begin{proof}
Note first that the tangent space of a set of the form $\pi^{-1}(\alpha(\DD))$ 
is the plane generated by the projection direction of $\pi$ and the tangent of the arc $\alpha$. Viewed in $\CC\PP^n$, it is the line joining the corresponding points.

After the blowing up of the origin of $\CC^{n+1}$, we get the strict transform of the set $\pi^{-1}(\alpha(\DD))$ and the hypersurface $(X,\bar{0})$. Their common intersection with the exceptional divisor is a set of isolated points by hypothesis. Around any of these points we have the intersection of a hypersurface and a surface that is not contained in it (because the projection is proper) in a smooth space. Then, the intersection is a (not necessarily irreducible) curve. Hence, for any point of $C_{X,\bar{0}}\cap L$ we have at least one component of the arc section.

\end{proof}

Note that the only case in which the hypothesis of the previous proposition are not satisfied is when $L$ is contained in $C_{X,\bar{0}}$.  

Note also that, if $C_{X,\bar{0}}$ intersects every line in more that one
point, no irreducible arc-sections can exist. This happens, for instance, if
$C_{X,\bar{0}}$ is a curve with degree bigger than three and no hyperflexes.

Then, we can summarize our conclusion in two propositions:

\begin{prop}\label{cor:ex} There exist irreducible hypersurface germs $(X,\bar{0})$ with all its arc-sections being reducible. Moreover, if we apply any analytic change of coordinates, the resulting germs also satisfy this property.
\end{prop}
\begin{cor}\label{cor:ex2} There exist  irreducible hypersurface germs $(X,\bar{0})$ such that the intersection with any smooth surface germ $(Y,\bar{0})$ is reducible.
\end{cor}
\begin{ex}\label{ex:ce0}
The smooth quartic given by $x^4 + y^4 + x^2z^2 + yz^3 + z^4=0$ does not have hyper-flexes, and then it is an explicit example. Note that the set of smooth quartics without hyperflexes is Zariski open in the set of  quartics. In this example, the equation considered in affine coordinates gives the germ in $(\mathbb{C}^3,\bar{0})$, and the same equation considered as homogenous coordinates in $\mathbb{CP}^2$ gives its tangent cone.
\end{ex}

\begin{ex}\label{ex:ce}
Another example is any surface with smooth irreducible tangent cone of degree greater or equal than 4 and no hyperflexes. For every degree this is a Zariski open set of the set of surfaces of that degree. Other examples are given by surfaces whose  tangent cone is the union of two non-tangent conics.
\end{ex}

\begin{prop}\label{prop:L}
A necessary condition for the existence of an irreducible arc-section in $(X,\bar{0})$ is that there exists a line $L$, which may be contained in one irreducible component of $C_{X,\bar{0}}$, that meets any component of $C_{X,\bar{0}}$ that does not contain it (if any) in exactly one point. This point would be an intersection point of all the irreducible components of $C_{X,\bar{0}}$, in case there are more than one.
\end{prop}
%
%
%
%
%
%
%
%
%

When this necessary condition is achieved, that is, in case  there exists such a line $L$, then the starting point to look for irreducible arcs sections is the following:
\begin{itemize}
\item[(i)] consider a projection $\pi$ in the direction of a point $p_\pi\in L$. Since we ask the projection to be proper, we have to require the line $<p_\pi>\subset\CC^{n+1}$ to not be contanined in $(X,\bar{0})$;
\item[(ii)] once the projection is fixed, the only arcs that may lead to irreducible arc sections must have tangent cone equal to $\overline{\pi}(L')$ with $L'$ some line as in Corollary \ref{prop:L} passing through the projection point. In Section \ref{sec:disc}, we give an alternative description of the possible tangent cones in some cases.
\end{itemize}
\begin{remark}\label{rem:clas_Surf}
We can easily summarize, in the surface case, the cases in which such a line $L$ can exist. 
There are 5 cases for $C_{X,\bar{0}}$ with reduced structure, that are (a) a line, (b) a union of lines meeting in a common point $p$, (c) a conic, (d) an irreducible surface of degree $d>2$ that  intersects some tangent line with multiplicity $d$ and  (e) a reducible surface with a common point $p$ of all its components (and such that there is a line through $p$ that is a tangent line of all the components of degree $d\geq 2$ of maximal intersection multiplicity).


We can saee that (d) is never the case if $C_{X,\bar{0}}$ is a generic variety of degree greater or equal than 4. 

\end{remark}

We retake the examples in the introduction, that cover the cases (a)-(b) in the previous remark:

\begin{ex} For the surface $z^d=f(x,y)$ with $ord(f)=m>d$ with $gcd(d,m)=1$. The tangent cone is $z=0$. Then, we can expect irreducible arc sections for any projection transverse to $z=0$ and in any direction. In particular for the projection in the direction of $z$ one can check easily  that $(at,bt)$ gives irreducible arc section whenever the initial form of $f$ evaluated at $(a,b)$ doesn't vanish.
\end{ex}
\begin{ex} In the case of $z^d=f(x,y)$ with $ord(f)=m<d$ and
 $gcd(d,m)=1$, we have that the tangent cone can be in general
given by a union of lines (in $\CC\PP^2$).
So, we can expect irreducible arc sections in any direction for a non-transverse projection. Then, for example, for the projection to the XY-plane, for generic $(a,b)$, it is easy to see that
$(at, bt)$ give irreducible arc sections.
 \end{ex}

Looking at the possibilities (a)-(e) for the case of surfaces in Remark \ref{rem:clas_Surf},  we see that the existence of irreducible arc-sections for the generic projection may only hold in the cases (a)-(c) and cases in $(e)$ where the tangent cone is a product of tangent cones as in $(a)-(c)$ (as for example $C_{X,\bar{0}}$ being a union of a conic and some lines through a given point). A natural question would be wether the existance of an irreducible arc-section in these cases is invariant by changing the generic projection or even by analytical change of coordinates.

\section{Existence of reducible arc-sections with $d$ components} \label{sec:red}

Let  $(X,\bar{0})$ be a surface germ in $(\CC^3,\bar{0})$ and let $\pi:\CC^3\to \CC^2$ be a projection so that $\pi|_X$ is finite. 

\begin{defn} We define transversal monodromies as follows:
\begin{itemize}
\item[(a)] we say that the monodromy of a loop $\delta$ parametrizing the boundary of a small disk transversal to a component of $\Delta$  is the \emph{transversal monodromy} associated to the component of ${\Delta}$ it meets.
\item[(b)]
Given any exceptional component $E_i$ of a composition of  blow ups over the origin of $(\CC^2,0)$, we  can consider any arc $\gamma$ in $\CC^2$ with transversal lifting through $E_i$ not meeting any other component of the total transform of $\Delta$. Given $\DD$ such that $\gamma(\DD)\cap \Delta=\{\bar{0}\}$, the monodromy of $\gamma|_{\partial \DD}$ is called the \emph{transverse monodromy} associated to the divisor $E_i$.
\end{itemize}
\end{defn}

Take an embedded resolution of $(\CC^2,\Delta)$. Let $\Delta^*$ be the total transform of $\Delta$.




Take a normal crossing point in $\Delta^*$. It has two local branches, $F_1$
and $F_2$. Let $\sigma_1$ and $\sigma_2$ be the permutations associated to its transversal monodromies. It is well known that meridians around two smooth curves as $F_1$ and $F_2$ meeting transversaly (that is, loops bounding small transversal discs to $F_i$) commute. Since the monodromy is a group morphism, $\sigma_1$ and $\sigma_2$ must commute too.

Now, if we make an extra blowup
at this point, a new divisor $F_3$ will appear, whose permutation corresponding to the transversal monodromy will be the composition of $\sigma_1$ and $\sigma_2$ (the order doesn't really matter, since they commute).

By successive blowing ups we can obtain divisors whose corresponding transversal permutation
is $\sigma_1^n\cdot \sigma_2^m$ for every positive integers $n,\ m$. When $n$ and $m$ are multiples
of the order of $\sigma_1$ and $\sigma_2$, we will obtain a divisor with trivial permutation.

So we have proven the following:
\begin{theorem}
\label{thm:existfullyreducedsection}
Let $\pi:\CC^3\to \CC^2$ be a linear projection such that $\pi|_{(X,\bar{0})}$ is finite. There always exist an arc $\gamma$ for which the curve $(X,\bar{0})\cap \pi^{-1}(\gamma(\DD))$
has as many irreducible components as the degree of the cover $\pi|_X$.
\end{theorem}
\begin{proof}
 Just take $\gamma$ to be the contraction of a smooth arc transverse to a divisor
 with trivial permutation.
\end{proof}

\begin{remark}
A divisor with trivial permutation can always be obtained by blowing up a generic
 point in any component of the exceptional divisor and then keep blowing up the point at the normal crossing of the divisors that appears. In that case the sequence of permutations that
 appear are consecutive powers of the original one.

This implies that there are arcs with totally reducible arc-sections with
arbitrary tangent cone.
\end{remark}

\begin{ex} For the case $z^d=f(x,y)$, if $x=0$ is a generic line for $f(x,y)=0$ (i.e. not tangent to it), for the projection to the $XY$- plane, the arc $(t^{d+1},t^d)$ gives a reducible arc-section.  In this case this is easy to check since $z^d-f(t^{d+1},t^d)=0$ has tangent cone $z^d-t^{kd}=0$ with $k$ the multiplicity of $f(x,y)=0$, which factorizes in $d$ factors.
\end{ex}


\subsection{The generic arc-section}\label{sec:gen}
Let  $(X,\bar{0})$ be a surface germ in $(\CC^3,\bar{0})$. For a given projection $\pi:\CC^3\to \CC^2$ with $\pi|_{X,\bar{0}}$ finite we will say that a \textit{generic arc} is a smooth arc not tangent to the discriminant $\Delta$ of $\pi|_{X,\bar{0}}$.

All generic arcs give the same associated monodromy. Moreover, this is the case for any projection in a direction of the open set $\Omega\subset Grass_1(\CC^3)$ of directions that are Zariski equisingularity. This is because we can deform one generic arc to the other. Then, one is transitive if and only if all of them are.

Given a singular arc transverse to the discriminant, then its monodromy is a power of the monodromy of the generic smooth arc by arguments in Section \ref{sec:red}. Then, if a singular arc gives rise to an irreducible arc section, its monodromy must be transtive and since it is the power of the monodromy of the generic arc, the monodromy of the generic arc should be also transitive.  That is:


\begin{cor}\label{cor:gen1}
If an arc not tangent to $\Delta$ gives an irreducible arc-section, then all smooth arcs not tangent to the discriminant do.
\end{cor}
\begin{cor}\label{cor:gen2} If the generic arc gives an irreducible arc section, then either the tangent cone $C_{X,\bar{0}}$ is a line (with certain multiplicity) or the tangent cone is a union of lines and we are projecting from the singular point of $C_{X,\bar{0}}$.
\end{cor}
\begin{proof} Looking at the cases $(a)-(e)$ in Remark~\ref{rem:clas_Surf} we see that these two cases are the only ones that allow irreducible arc sections in any generic direction.
\end{proof}

\begin{cor}\label{cor:gen3}
If $C_{X,\bar{0}}$ is not a line or a union of lines and there is an irreducible arc-section over an arc $\gamma$, then the arc $\gamma$ is tangent to $\Delta$.
\end{cor}
%
%
%
%
%
%

\section{Discriminant of the tangent cone coincides with the tangent cone of the discriminant}\label{sec:disc}
After Corollary \ref{cor:gen3} we are interested in the tangent cone of $\Delta$ in order to know the tangent line of an arc with irreducible-arc section.

Let $(X,\bar{0})$ be a surface in $(\CC^3,0)$. We want to prove that when projecting in the direction of a point in a line $L$ as in the hypothesis of Lemma \ref{lem:L}, we can know in advance the only possible tangent line for an arc with irreducible arc-section.

Following \cite{Tei}, we consider the Nash modification $\widetilde{X}$ that is the closure in $\CC^3\times Grass_2(\CC^3)\approx \CC^3\times \check{\CC\PP^2}$ of $\{(x,T_xX): x\in X\}$. The projection of $\widetilde{X}$ onto the first factor is called the Nash blowup $$\calN: \widetilde{X}\to X$$ and the one onto the second, the gauss map
$$\widetilde{\calG}:\widetilde{X}\to \check{\CC\PP^2}.$$

We recall that $\calN^{-1}(0)\subset \{0\}\times \check{\CC\PP^2}$ is the union of $\check{C}_{X,\bar{0}}$, that denotes the dual curve of the tangent cone $C_{X,\bar{0}}\subset \CC\PP^2$, and some lines that correspond to  pencils of planes passing through the exceptional tangent lines.

Every point in $\calN^{-1}(0)\subset \{0\}\times \check{\CC\PP^2}$ corresponds to a hyperplane in $(\CC^3,0)$, limit of tangent hyperplanes to $(X,\bar{0})$. In particular, a point in the dual of the tangent cone, represents the plane in $\CC^3$ induced by the tangent line to its dual point in $C_{X,\bar{0}}\subset \CC\PP^2$.

\begin{prop}\label{prop:disc} Assume the projection direction is not an exceptional tangent. Then, the tangent of the discriminant  $ \Delta$ coincides with the discriminant of $\overline{\pi}|_{C_{X,\bar{0}}}$.
\end{prop}

\begin{proof}
The projection direction $p_\pi$ gives a line $T_\pi$ in $\check{\CC\PP^2}\approx\CC\PP^2$ that corresponds to the pencil of hyperplanes that contains the direction.

The set $\calN(\calG^{-1}(T_\pi))$ gives the polar curve $\Gamma$ of the projection. Since $T_\pi$ is not the pencil of an exceptional tangent, we have that  $\calN(\calG^{-1}(T_\pi))$ has one dimensional component in $(X,\bar{0})$. This is because $T_\pi$ is not the pencil of an exceptional tangent line and because $\calG^{-1}(T_\pi)$ is defined by $1$ equation in a surface $\widetilde{X}$ and $\calG^{-1}(T_\pi)$ is not contained in $\calN^{-1}(0)$.

It is immediate that whenever we have a point in the discriminant of $\bar{\pi}$, that is, whenever there exists $q\in C_{X,\bar{0}}$ such that $\overline{qp}_\pi$ is a line tangent to $C_{X,\bar{0}}$ at $q$, we have an intersection point of $T_\pi\cap \check{C}_{X,\bar{0}}$, that is a limit tangent hyperplane to $(X,\bar{0})$ that contains $<p_\pi>$. And reciprocally, an intersection point of $T_\pi\cap \check{C}_{X,\bar{0}}$ gives a line tangent to $C_{X,\bar{0}}$ passing by $p_\pi$.


Finally, the tangent line at the origin of every branch of the discriminant $\pi(\Gamma)$ is not zero and is given by limits of tangent lines out of the origin. These lines are the image by $\overline{\pi}$ of the tangent lines of the polar curve, which are contained in the hyperplanes corresponding to points of $T_\pi$. Then, since the tangent line is not zero, it has to be in the discriminant or $\overline{\pi}$.

\end{proof}

\begin{cor}\label{cor:disc} Assume all the points on $C_{X,\bar{0}}$ whose tangent line is as in Lemma \ref{lem:L}, are not exceptional tangent lines.

Assume $C_{X,\bar{0}}$ is not a line or a union of lines.

Then, any arc with irreducible arc-section has as tangent line a point in the discriminant of $\overline{\pi}|_{C_{X,\bar{0}}}$.
\end{cor}

Examples where the hypothesis of Corollary \ref{cor:disc} are satisfied are any surface with $C_{X,\bar{0}}$ smooth, since it is known that there is no exceptional tangents in this case (see (2.2.1) in \cite{Tei}).






\section{Algorithm to find an irreducible arc-section.}\label{sec:alg}
Let  $(X,\bar{0})$ be a surface germ in $(\CC^3,\bar{0})$.
Based on the previous sections, we get a general method to determine if a given
surface germ projection $\pi:(X,\bar{0})\to(\CC^2,0)$ admits irreducible arc sections or not.

First of all, we find the discriminant $ \Delta$.
The following steps are:
\begin{enumerate}

\item Compute an embedded resolution of singularities of $\Delta$,
$\phi:(\widetilde\CC^2,E)\to(\CC^2,\bar{0})$.
Applying the same sequence of blowups, we get a map $\widetilde{\phi}:(\widetilde X,\bar{E})\to(X,\bar{0})$.
\item\label{itempardermuts} For each normal crossing along the total transform of $\Delta$, that is, for every intersection of two branches
$B_1,B_2$ in $\phi^{-1}(\Delta)$, get two small smooth disks $D_1,D_2$ such that:
\begin{itemize}
\item $D_i$ is centered in a point of $B_i$, is transversal to it, and touches no other component of
$\phi^{-1}(\Delta)$
\item There exists a point $q$ in the intersection of the boundaries
$\partial D_1\cap \partial D_2$
\item Both disks are inside a small polydisc centered in the normal crossing
point, that meets no other component of $\phi^{-1}(\Delta)$.
\end{itemize}
 \item Consider $\pi\circ\widetilde\phi$ restricted to the preimage of
 $\partial D_1$ and   $\partial D_2$ as in \ref{itempardermuts}, and compute the permutations $P_1$ and $P_2$ of the corresponding monodromy based in $q$.
 \item The projection $\pi$ allows an irreducible arc section if and only if for some
 $P_1,P_2$ as before, the group generated by them contains a transitive element.

\end{enumerate}

\begin{figure}
\begin{center}
\begin{tikzpicture}[samples=800,scale=7]

 \draw[domain=-0.4:0.4,thick] plot ({\x * \x * \x},{\x * \x}) node[above] {$(y^3-x^2)(y^3+x^2)$};
 \draw[domain=-0.4:0.4,thick] plot ({  \x * \x * \x},{-\x * \x});

 \draw[-latex,ultra thick] (0,-0.5) -- (0,-0.3) node[midway,right]  {$(x,y)\mapsto (xy,x)$};

 \begin{scope}[yshift=-0.8cm]
 \draw[domain=-0.1:0.1,thick] plot ({\x },{10*\x * \x}) node[above] {$y^4(y-x^2)(y+x^2)$};
 \draw[domain=-0.1:0.1,thick] plot ({  \x},{-10*\x * \x});
 \draw[domain=-0.1:0.1,very thick] plot({\x},{0});
 \end{scope}

 \draw[-latex,ultra thick] (0,-1.2) -- (0,-1) node[midway,right]  {$(x,y)\mapsto (x,xy)$};

 \begin{scope}[yshift=-1.5cm]
 \draw[domain=-0.1:0.1,thick] plot ({\x },{\x}) node[above] {$y^4x^6(y-x)(y+x)$};
 \draw[domain=-0.1:0.1,thick] plot ({  \x},{- \x});
 \draw[domain=-0.1:0.1, thick] plot({\x},{0});
 \draw[domain=-0.1:0.1,very thick] plot({0},{\x});
 \end{scope}

 \draw[-latex,ultra thick] (0.35,-1.9) -- (0.2,-1.7) node[midway,right]  {$(x,y)\mapsto (x,xy)$};

 \begin{scope}[xshift=0.5cm, yshift=-2.2cm]
 \draw[domain=-0.1:0.1,thick] plot ({\x },{0.05}) node[above] {$y^4x^{12}(y-1)(y+1)$};
 \draw[domain=-0.1:0.1,thick] plot ({  \x},{- 0.05});
 \draw[domain=-0.1:0.1, thick] plot({\x},{0});
 \draw[domain=-0.1:0.1,very thick] plot({0},{\x});
 \end{scope}

 \draw[-latex,ultra thick] (-0.35,-1.9) -- (-0.2,-1.7) node[midway,left]  {$(x,y)\mapsto (xy,y)$};

 \begin{scope}[xshift=-0.5cm, yshift=-2.2cm]
 \draw[domain=-0.1:0.1,thick] plot ({0.05},{\x}) node[above] {$y^{12}x^{6}(1-x)(1+x)$};
 \draw[domain=-0.1:0.1,thick] plot ({ -0.05},{\x});
 \draw[domain=-0.1:0.1,very thick] plot({\x},{0});
 \draw[domain=-0.1:0.1, thick] plot({0},{\x});
 \end{scope}

 \end{tikzpicture}

 \caption{\label{fig:blowups}Sequence of blowups to resolve the singularities of the discriminant.}
 \end{center}
 \end{figure}
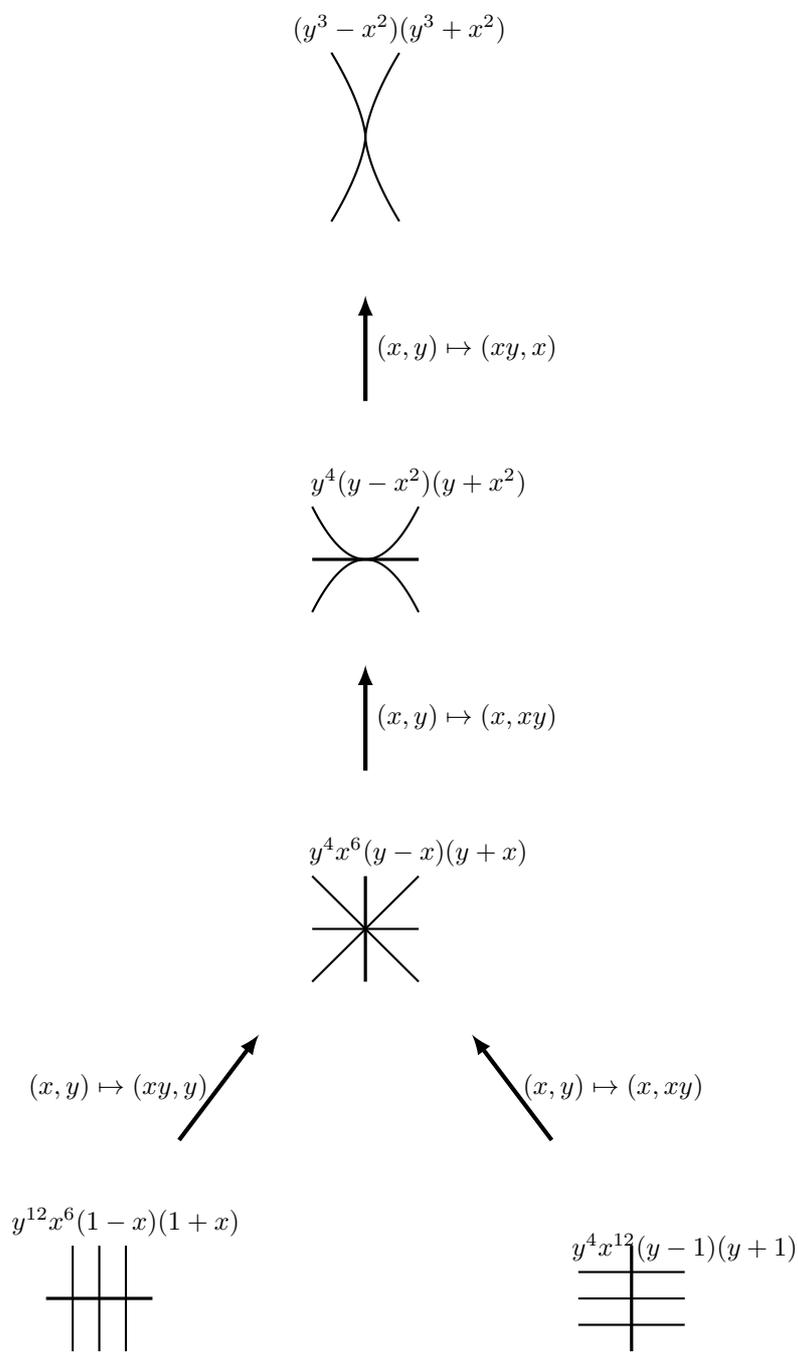

\begin{ex}
\label{ex:method}
Consider the surface germ given by $z^4-4xz+3y^2=0$, projected along the $z$
direction. Note that it is an isolated singularity, and hence it must be irreducible. Its discriminant is given by the equation $(y^3-x^2)(y^3+x^2)=0$.

We compute a resolution of singularities by the sequence of blowups in figure~\ref{fig:blowups}.

We get three exceptional divisors and two branches. Let's focus for instance in
the normal crossing at the origin of the last chart. The two discs $B_1,B_2$ can
be taken as $x=\frac{1}{2},\|y\|\leq \frac{1}{2}$ and $\|x\|\leq\frac{1}{2},y = \frac{1}{2}$ respectively, so the base point $q$ would be $(\frac{1}{2},\frac{1}{2})$. The composition of the blow-up transformations applied to
the equation of $X$ is $z^4-4(x^2y^3)z+3(xy^2)^2=0$.

To compute the permutation around the boundary of the disk, we parametrize the disk as $(\frac{1}{2},\frac{t}
{2})$ and substitute, obtaining $z^4-\frac{1}{8}t^3z+\frac{3}{64}t^4$. This is a
squarefree homogenous equation in two variables, so it factors as the product of
four linear forms. That is, the arc section corresponding to this arc is a union
of four lines, and the corresponding monodromy has a trivial permutation.

The other disk can be parametrized as $(\frac{t}{2},\frac{1}{2})$ and after composing, we get the equation $z^4 - \frac{1}{8}t^2z + \frac{3}{64}t^2=0$. The
braid monodromy of this curve around the boundary of the unit disk is given by the braid shown in Fig~\ref{fig:braid}.
\begin{figure}[ht]
\begin{center}
\begin{tikzpicture}[samples=800,scale=0.6]

\def\sigma{

    \draw[thick] (0,0) to [out=90,in=270] (1,1);
    \draw[thick] (1,0) to [out=90,in=325] (0.7,0.4);
    \draw[thick] (0.3,0.6) to [out=145,in=270] (0,1);
    }

\def\s0s2s1{
    \sigma
    \begin{scope}[xshift=2cm]
    \sigma
    \end{scope}

    \begin{scope}[xshift=1cm,yshift=1cm]
    \sigma
    \end{scope}
    \draw[thick] (0,1) to (0,2);
    \draw[thick] (3,1) to (3,2);
    }
\s0s2s1
\begin{scope}[yshift=2cm]
\s0s2s1
\end{scope}

 \end{tikzpicture}
\end{center}
\caption{The braid $(\sigma_0\sigma_2\sigma_1)^2$\label{fig:braid}}
\end{figure}
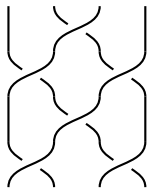

This braid can be expressed in the Artin presentation (see for instance~\cite{artin}) as $(\sigma_0\sigma_2\sigma_1)^2$. Its permutation is $(1,4)(2,3)$. In this case,
the group generated by these two permutations only has two elements, none of which
is transitive.

We can repeat these computations with the rest of the crossings. In this case, one
of the permutations will always be trivial (since they correspond to the first
divisor we have computed). Similar computations show that the rest of the
permutations we obtain are $(3,4)$, $(3,4)$ and $(1,4,2)$ respectively. None of
them are transitive and the group generated by them and the trivial permutation
will not contain any transitive element. That is: in this case, there is no
irreducible arc section.

\end{ex}

In fact, not every possible branch must be checked. It is enough to study the branches $\Delta_i$ of $\Delta$ whose tangent line is contained in a line $L$ as in Lemma \ref{lem:L}. Moreover, we have a necessary condition:
\begin{lemma}
\label{lemma:deformationtoarc}
Consider an arc that can be deformed to a branch $\Delta_i$ of the discriminant without
intersecting the discriminant outside $\Delta_i$, and whose corresponding arc-section is irreducible, then a parametrizaton of $\Delta_i$ also has an irreducible arc-section.
\end{lemma}
\begin{proof}
Let $\gamma_i:(\mathbb{D},0)\to (\mathbb{C}^2,0)$ be such an arc, and
$\delta:(\mathbb{D},0)\to (\mathbb{C}^2,0)$ be a parametrization of $\Delta_i$.
Take a continuous deformation $T:(\mathbb{D},0)\times I\to (\mathbb{C}^2,0)$
such that $T\mid_{(\mathbb{D},0)\times\{0\}}=\gamma_i$, $T\mid_{(\mathbb{D},1)\times\{0\}}=\delta_i$,
and $T(\partial\mathbb{D}\times [0,1))\cap \Delta=\emptyset$.

The restriction of this deformation to the boundary of $\mathbb{D}$ gives a
continuous deformation from the closure of
the braid corresponding to $\gamma_i$ to the closure of the braid corresponding
to $\Delta_i$. Since the arc section of $\gamma_i$ is irreducible, its braid is
transitive, or equivalently, its closure is a knot. Since the deformation of the
arcs does not intersect the discriminant, the different steps of the deformation
(except for the last one) are knots isotopic to the original one. So, the closure
of the braid corresponding to $\Delta_i$ is the image of a knot by a continuous function, and hence, it must have also only one connected component. Hence the
braid corresponding to $\Delta_i$ must also be transitive.
\end{proof}

This phenomenon happens, for instance, with the arcs that are parallel to the
branch $\Delta_i$ in the chart centered at the normal crossing between $\Delta_i$
and the rest of the exceptional divisor. So if we are interested in finding
irreducible arc sections, in the previous method we can skip the cases around
branches of the discriminant whose parametrization doesn't give an irreducible
arc section.

\subsection{Sections over arcs contained in the discriminant}
We finish with an example where the behaviour over a parametrization of a branch
of the discriminant changes under a small change of the projection.

\begin{ex}
Consider the surface germ given by $z^3-(x-y)(x+y)(x-2y)(x+2y)$. If we project
in the $z$ direction, the discriminant is formed by the four lines $x=y$,
$x=-y$, $x=2y$ and $x=-2y$.
Over a generic arc such as $x=0$ we get an irreducible arc-section (with equation $z^3-y^4$).

Over each point of the discriminant, there is only one point of the surface. So, if we take an arc that parametrizes one component of the discriminant, the arc section that we get is a multiple line.

Now consider a line that doesn't go through the origin, say $x=\epsilon$
for some small $\epsilon>0$. It meets transversally the discriminant in four
points (one for each component) $(\epsilon,\epsilon), (\epsilon,-\epsilon),
(\epsilon,2\epsilon), (\epsilon,-2\epsilon)$. Over each of these four
points, there is only one preimage. It is easy to check that the local
equation for the arc section over the arc that parametrizes the line
$x=\epsilon$ centered at any of these four points is a vertical flex.


\begin{figure}[th]
    \begin{center}
\begin{tikzpicture}[y=0.80pt, x=0.80pt, yscale=-1.30000, xscale=1.30000, inner sep=-2pt, outer sep=-2pt]
  \path[draw=black,line join=miter,line cap=butt,even odd rule,line width=0.212pt]
    (41.5774,93.6488) .. controls (43.4673,92.5149) and (151.9464,93.6488) ..
    (151.9464,93.6488) -- (177.0818,68.5134) -- (65.7679,68.5134) --
    (41.8609,92.4204);
  \path[draw=black,line join=miter,line cap=butt,even odd rule,line width=0.265pt]
    (76.9292,88.3251) -- (123.5871,72.1362);
  \path[draw=black,line join=miter,line cap=butt,even odd rule,line width=0.212pt]
    (112.6369,69.4583) -- (97.1399,89.4911);
  \path[draw=black,line join=miter,line cap=butt,even odd rule,line width=0.212pt]
    (100.5417,69.0804) -- (112.6369,89.1131);
  \path[draw=black,line join=miter,line cap=butt,even odd rule,line width=0.212pt]
    (91.4702,72.1042) -- (126.6220,87.2232);
  \path[draw=black,line join=miter,line cap=butt,even odd rule,line width=0.212pt]
    (74.4613,86.0893) -- (137.5833,86.0893);
  \path[draw=black,line join=miter,line cap=butt,even odd rule,line width=0.212pt]
    (72.9494,52.7329) .. controls (72.1935,51.9769) and (137.2054,52.7329) ..
    (137.2054,52.7329) -- (137.2054,4.3519) -- (73.3274,4.3519) -- cycle;
  \path[draw=black,dash pattern=on 0.21pt off 1.69pt,line join=miter,line
    cap=butt,miter limit=4.00,even odd rule,line width=0.212pt] (83.6272,85.9003)
    -- (83.6272,26.0856);
  \path[draw=black,dash pattern=on 0.21pt off 1.69pt,line join=miter,line
    cap=butt,miter limit=4.00,even odd rule,line width=0.212pt] (99.7857,85.9003)
    -- (99.7857,26.1801);
  \path[draw=black,dash pattern=on 0.21pt off 1.27pt,line join=miter,line
    cap=butt,miter limit=4.00,even odd rule,line width=0.212pt] (110.7470,85.9003)
    -- (110.7470,26.1801);
  \path[draw=black,dash pattern=on 0.21pt off 1.27pt,line join=miter,line
    cap=butt,miter limit=4.00,even odd rule,line width=0.212pt] (123.9762,86.2783)
    -- (123.9762,26.5580);
  \path[draw=black,line join=miter,line cap=butt,even odd rule,line width=0.212pt]
    (73.8943,35.8185) .. controls (80.8869,37.6138) and (83.8162,31.3772) ..
    (83.6272,26.0856) .. controls (83.4382,20.7939) and (87.0290,15.3132) ..
    (93.6436,17.2031);
  \path[draw=black,line join=miter,line cap=butt,even odd rule,line width=0.212pt]
    (93.1904,33.7287) .. controls (97.8206,34.7681) and (99.8995,31.5553) ..
    (99.7105,26.2637) .. controls (99.5215,20.9720) and (101.8838,19.2711) ..
    (104.9076,18.0427);
  \path[draw=black,line join=miter,line cap=butt,even odd rule,line width=0.212pt]
    (104.8131,32.3113) .. controls (107.4590,31.9333) and (110.8608,31.0829) ..
    (110.6718,25.7912) .. controls (110.4828,20.4995) and (113.2231,20.4050) ..
    (115.3020,19.8381);
  \path[draw=black,line join=miter,line cap=butt,even odd rule,line width=0.212pt]
    (118.6093,33.8232) .. controls (121.2551,33.4452) and (124.0899,30.2324) ..
    (123.9009,24.9407) .. controls (123.7120,19.6491) and (126.3578,19.0821) ..
    (128.1532,17.8537);
  \path[draw=black,line join=miter,line cap=butt,even odd rule,line width=0.212pt]
    (42.8691,225.6095) .. controls (44.7590,224.4755) and (153.2382,225.6095) ..
    (153.2382,225.6095) -- (178.3736,200.4740) -- (67.0596,200.4740) --
    (43.1526,224.3810);
  \path[draw=black,line join=miter,line cap=butt,even odd rule,line width=0.212pt]
    (75.7531,218.0499) -- (138.8751,218.0499);
  \path[draw=black,line join=miter,line cap=butt,even odd rule,line width=0.212pt]
    (74.2412,184.6935) .. controls (73.4852,183.9376) and (138.4971,184.6935) ..
    (138.4971,184.6935) -- (138.4971,136.3126) -- (74.6191,136.3126) -- cycle;
  \path[draw=black,dash pattern=on 0.20pt off 1.23pt,line join=miter,line
    cap=butt,miter limit=4.00,even odd rule,line width=0.204pt] (87.2813,217.8609)
    -- (87.2813,156.5815);
  \path[draw=black,dash pattern=on 0.22pt off 1.30pt,line join=miter,line
    cap=butt,miter limit=4.00,even odd rule,line width=0.217pt]
    (101.0775,217.8610) -- (101.0775,155.3059);
  \path[draw=black,dash pattern=on 0.21pt off 1.27pt,line join=miter,line
    cap=butt,miter limit=4.00,even odd rule,line width=0.212pt]
    (113.0037,217.8610) -- (113.0037,158.1407);
  \path[draw=black,dash pattern=on 0.21pt off 1.27pt,line join=miter,line
    cap=butt,miter limit=4.00,even odd rule,line width=0.212pt]
    (127.6303,218.2389) -- (127.6303,158.5187);
  \path[draw=black,line join=miter,line cap=butt,even odd rule,line width=0.212pt]
    (82.0208,221.4048) .. controls (83.7690,218.0502) and (106.8728,210.3490) ..
    (106.8728,210.3490) .. controls (106.8728,210.3490) and (78.9970,220.1291) ..
    (76.9182,220.5543);
  \path[draw=black,dash pattern=on 0.21pt off 1.23pt,line join=miter,line
    cap=butt,miter limit=4.00,even odd rule,line width=0.206pt] (84.4304,217.9557)
    -- (84.4304,161.4483);
  \path[draw=black,line join=miter,line cap=butt,even odd rule,line width=0.212pt]
    (99.7857,221.8772) .. controls (100.2270,215.9711) and (106.9673,210.3490) ..
    (106.9673,210.3490) .. controls (106.9673,210.3490) and (96.9981,219.6566) ..
    (94.9665,220.1763);
  \path[draw=black,line join=miter,line cap=butt,even odd rule,line width=0.212pt]
    (92.9936,174.2694) .. controls (95.6394,173.8914) and (101.1768,172.1917) ..
    (101.0256,166.8989) .. controls (100.9311,163.5916) and (98.5573,164.6138) ..
    (98.3683,160.6451) .. controls (98.5573,156.8181) and (104.2695,153.0558) ..
    (105.0255,152.6778);
  \path[draw=black,dash pattern=on 0.21pt off 1.24pt,line join=miter,line
    cap=butt,miter limit=4.00,even odd rule,line width=0.207pt] (98.3683,218.0030)
    -- (98.3683,160.6451);
  \path[draw=black,line join=miter,line cap=butt,even odd rule,line width=0.212pt]
    (114.6213,221.4992) .. controls (114.5107,218.5696) and (106.9673,210.3490) ..
    (106.9673,210.3490) .. controls (106.9673,210.3490) and (116.7946,219.7983) ..
    (118.2593,220.2236);
  \path[draw=black,dash pattern=on 0.22pt off 1.32pt,line join=miter,line
    cap=butt,miter limit=4.00,even odd rule,line width=0.219pt]
    (115.3772,217.8140) -- (115.3772,153.9360);
  \path[draw=black,line join=miter,line cap=butt,even odd rule,line width=0.212pt]
    (107.4673,171.7650) .. controls (110.1131,171.3870) and (115.6505,169.6874) ..
    (115.4992,164.3945) .. controls (115.4047,161.0872) and (113.0310,162.1095) ..
    (112.8420,158.1407) .. controls (113.0310,154.3137) and (118.7432,150.5514) ..
    (119.4992,150.1734);
  \path[draw=black,line join=miter,line cap=butt,even odd rule,line width=0.212pt]
    (131.9137,221.4048) .. controls (132.7749,218.5069) and (106.9673,210.3490) ..
    (106.9673,210.3490) .. controls (106.9673,210.3490) and (132.2458,219.0100) ..
    (137.2512,219.1819);
  \path[draw=black,dash pattern=on 0.22pt off 1.32pt,line join=miter,line
    cap=butt,miter limit=4.00,even odd rule,line width=0.220pt]
    (131.9137,217.8140) -- (131.9137,153.6525);
  \path[draw=black,line join=miter,line cap=butt,even odd rule,line width=0.212pt]
    (124.5753,171.7476) .. controls (127.2212,171.3696) and (132.3238,171.0861) ..
    (132.1348,165.7944) .. controls (131.9459,160.5028) and (127.8193,162.4874) ..
    (127.6303,158.5187) .. controls (127.8193,154.6917) and (135.8513,150.5339) ..
    (136.6072,150.1560);
  \path[draw=black,line join=miter,line cap=butt,even odd rule,line width=0.212pt]
    (81.3201,180.3836) .. controls (88.3849,178.8980) and (88.2508,172.0091) ..
    (85.7388,167.3478) .. controls (83.2268,162.6865) and (84.0204,156.1826) ..
    (90.7881,154.9490);

\end{tikzpicture}
\end{center}
\caption{A rotation causes the discriminant to unfold from lines to cusps.}
\end{figure}
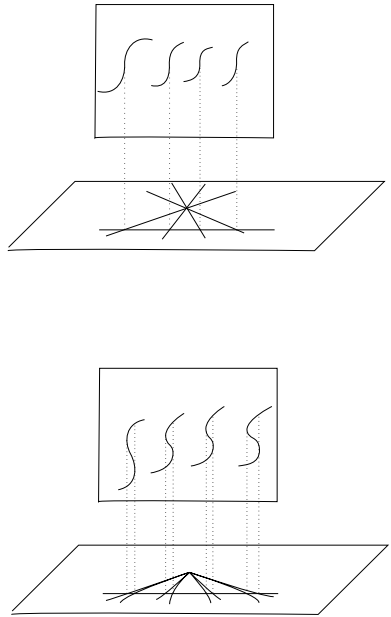

Now we make a small deformation of the projection (e.g. make the change
of variables $y=y+\delta z$ for some small $\delta>0$).
Then the four lines of the discriminant get deformed into $4$ cusps.
Let's see the arc section for an arc that parametrizes one of them. To do
so, consider again the section over the line $x=\epsilon$. It is the same
small change of variables aplied to the previous curve with four vertical
flexes. If we rotate slightly a vertical flex, it splits in two different
simple vertical tangencies. That means that over each point of the
discriminant, there will be two points of the surface: one with multiplicity
$2$ and other
with multiplicity $1$. Hence, the arc section corresponding to a
parametrization of one of the cusps will contain a component with multiplicity $2$
and another with multiplicity $1$. That is, it will be reducible.

Note that both projections are non-transversal to the surface and have different discriminants but the generic arc section is irreducible in both cases.

\end{ex}

If we apply the previous algorithm to this example, we can find a
smooth arc whose arc section has $3$ components, just by making it have
contact $3$ with one of the lines of the discriminant. An explicit equation could be $(t-t^3,t+t^3)$; then the equation for the arc section is $-12 t^{10} - 40 t^{8} - 12 t^{6} + z^{3}$, which is analytically equivalent to $ - 12 t^{6} + z^{3}$. Clearly this factorizes as the product of three components.

\bibliographystyle{unsrt}
\bibliography{biblio}
%

\end{document}